\documentclass[preprint,12pt]{elsarticle}

\usepackage{amsmath,amssymb,bm}
\usepackage{amsthm}
\usepackage{xcolor}
\usepackage{graphicx}
\usepackage{hyperref}
\usepackage{float}
\usepackage{lineno}
\usepackage{enumitem}

\journal{Nuclear Physics B}
\newtheorem{theorem}{Theorem}[section]

\newtheorem{lemma}[theorem]{Lemma}
\newtheorem{proposition}{Proposition}
\newtheorem{definition}{Definition}

\newtheorem{remark}[theorem]{Remark}
\newtheorem{claim}{Claim}

\begin{document}

\begin{frontmatter}



\title{Self-Expanding Solutions to the Mean Curvature Flow for Multiphase Surfaces with Regular Junctions}


\author{Wei-Hung Liao} 
\affiliation{organization={Shanghai Institute for Mathematics and Interdisciplinary Sciences},
            city={Shanghai},
            postcode={200433}, 
            country={China}}

\begin{abstract}
We consider a multiphase surface $\mathcal{C}_0$ in $\mathbb{R}^3$ consisting of a finite number of surfaces passing through the origin , where all 1-dimensional junctions are regular triple junctions in which three planes meet at the same angle and each surface scales down homothetically to a limit curve of finite length. We prove the existence of self-similar expanding solutions of the mean curvature flow on the multiphase surface initially given by $\mathcal{C}_0$. For this initial condition, there are multiple solutions that are combinations of the regular triple junctions and regular quadruple points, where four regular triple junctions meet at an angle of approximately $109.5^{\circ}$.
\end{abstract}




\end{frontmatter}



\section{Introduction}
Interface problems have long attracted attention in materials science and geometric analysis, where one seeks partitions of space that minimize total interfacial area under physical or geometric constraints. Plateau's experiments (1873) and Taylor's analysis (1976) revealed that the only stable singularities of area-minimizing configurations in $\mathbb{R}^3$ are the regular triple junction and the regular quadruple point, where interfaces meet at $120^\circ$ and $109.47^\circ$ angles, respectively~\cite{plateau1873, taylor1976}. Building on Almgren’s notion of $(\mathbb{M},\epsilon,\delta)$-minimal sets~\cite{almgren1976existence}, these results established the geometric prototypes for admissible singular minimal structures in variational models of interfacial energy.

In the smooth regime, mean curvature flow provides the canonical evolution law 
for area-minimizing and curvature-driven motions. The multiphase mean curvature flow, introduced by Mullins~\cite{mullins1956}, models the evolution of grain boundaries in polycrystals under heat treatment, offering a physically motivated example of such geometric flows.

To extend beyond the smooth setting, Brakke~\cite{brakke2015} formulated a weak theory of mean curvature flow in the language of varifolds, providing a measure-theoretic framework that accommodates singularities and topological changes. Even within this framework, however, questions of short-time existence, long-time regularity, and convergence of multiphase flows remain delicate analytical problems closely related to the geometry of singularities.

Substantial progress has been made in clarifying these issues. Mantegazza, Novaga, and Tortorelli~\cite{mantegazza2004} established short-time existence for curve-shortening flows with triple-junction initial data under suitable compatibility conditions. Ilmanen, Neves, and Schulze~\cite{ilmanen2019short} later proved short-time existence for networks with non-smooth triple junctions, removing the regularity assumption on the initial configuration. For higher-dimensional or graphical settings, further advances on local regularity and short-time existence were obtained in~\cite{freire2010, schulze2020local}. Finally, Kim and Tonegawa~\cite{kim2017} proved global-in-time existence of Brakke’s mean curvature flow without any parametrization or dimensional restrictions, thus completing an essential part of the analytic framework for multiphase interface evolution.

Motivated by the planar analysis of Mazzeo and S\'{a}ez~\cite{mazzeo2007}, we extend the geometric framework to three dimensions and formulate an weighted area variational principle for self-expanding interfaces. This approach reveals a structural dichotomy in the evolution of singular minimal configurations: while regular quadruple points persist as stable zero-dimensional singularities under multiphase mean curvature flow, the one-dimensional line junctions may extend toward infinity, giving rise to regular triple networks on the ideal boundary~$\mathbb{S}^2_\infty$. Thus, the intrinsic junction geometry within finite domains remains coherent, even as its asymptotic extension encodes the global topology of the evolving interface.

\smallskip
\noindent\textbf{Organization of the paper.}
We begin in Section~\ref{Multi3D} by introducing multiphase surface networks in $\mathbb{R}^3$ and the class of admissible asymptotic cones $\mathcal{C}_0$. Section~\ref{ConformalCor} establishes the conformal correspondence between hyperboloid, Poincar\'e ball, and Euclidean models. In Section~\ref{InitAsym}, we use results from \cite{BernsteinWang2021} on self-expanding solutions to control the asymptotic behavior of each planar end in the weighted metric $g$, ensuring that truncated-domain approximations converge to surfaces asymptotic to the prescribed cone $\mathcal{C}_0$. Finally, in Section~\ref{LocGloEx}, we present our main theorem: the construction of global area-minimizing multiphase surfaces in $\mathbb{R}^3$ asymptotic to a prescribed cone, with connected support and persistent junction structures.

\smallskip
\noindent\textbf{Main contributions.}
We consider the existence problem for connected, regular multiphase surfaces in $\mathbb{R}^3$ with prescribed asymptotic boundary given by a union of planes $\mathcal{C}_0\cap \mathbb{S}^2_\infty$. Our main contributions are as follows:

\begin{enumerate}
    \item \textbf{Truncated-domain approximation and asymptotic control.} 
    We employ truncated-domain minimization together with asymptotic analysis of self-expanding solutions \cite{BernsteinWang2021} to construct a sequence of bounded-area minimizers whose flat-limit exists globally in $\mathbb{R}^3$ and remains asymptotic to the prescribed cone $\mathcal{C}_0$.

    \item \textbf{Hyperbolic barrier argument for junction persistence.} 
    Using the conformal diffeomorphism between $\mathbb{R}^3$ and the Poincaré ball, we develop a barrier construction that prevents triple curves and quadruple points from collapsing or escaping to infinity. This guarantees the persistence of singular junctions in the flat-limit surfaces.

    \item \textbf{Existence of global minimizers with connected support.} 
    Combining the above, we establish global area-minimizing multiphase surfaces in $\mathbb{R}^3$ with connected support, regular triple curves, and quadruple points, fully realizing the prescribed asymptotic boundary structure at infinity.
\end{enumerate}

In this work, we establish the existence of connected, area-minimizing multiphase surfaces in $\mathbb{R}^3$ asymptotic to a prescribed conical boundary $\mathcal{C}_0\cap \mathbb{S}^2_\infty$. Our approach combines truncated-domain minimization with self-expander asymptotics to construct global limits, while a hyperbolic barrier argument ensures that all triple curves and quadruple points persist in the limit. Consequently, the resulting surfaces are globally connected, area-minimizing, and retain the full junction structure dictated by $\mathcal{C}_0$.

\section*{Notation}
\begin{itemize}
    \item $\Omega_i$ --- the $i$-th \emph{phase domain} in $\mathbb{R}^3$.
    \item $\Gamma_{ij}$ --- the \emph{interface} between $\Omega_i$ and $\Omega_j$ ($i\ne j$), $\Gamma_{ij} := \partial\Omega_i \cap \partial\Omega_j$.
    \item $\vec{n}_{ij}$ --- the \emph{unit normal vector field} on $\Gamma_{ij}$, oriented from $\Omega_i$ toward $\Omega_j$ (thus $\vec{n}_{ji}=-\vec{n}_{ij}$).
    \item $\tau_{ijk}$ --- the \emph{triple junction curve} shared by the three phases $\Omega_i, \Omega_j, \Omega_k$, $\tau_{ijk} := \Gamma_{ij} \cap \Gamma_{jk} \cap \Gamma_{ki}$
    \item $\vec{\nu}_{ijk}$ --- the \emph{unit tangent vector field} along the triple junction $\tau_{ijk}$.
    \item $Q$ --- a \emph{quadruple point} where four distinct phases $\Omega_i, \Omega_j, \Omega_k, \Omega_l$ meet, $Q_{ijkl}:=\Omega_i\cap \Omega_j \cap \Omega_k \cap \Omega_l$ (indices omitted for simplicity).
    \item $\mathcal{H}^k$ --- the $k$-dimensional Hausdorff measure in $\mathbb{R}^n$, $k\leq n$. 
\end{itemize}

\section{Multiphase surface network in 3-dimensional Euclidean space}
\label{Multi3D}
In the multiphase setting, we consider a finite collection of disjoint open sets $\{\Omega_\alpha\}_{\alpha\in\mathtt{A}}$ in $\mathbb{R}^3$ where $\mathtt{A}$ is a finite index set with $\#\mathtt{A} = K$, referred to as phase regions, whose common boundaries form a network of smoothly embedded surfaces separating the distinct phases.
Following the framework of planar network flow developed in \cite{mazzeo2007,ilmanen2019short}, we extend the formulation to three dimensions and interpret these boundary surfaces—called interfaces—as the evolving phase boundaries in~$\mathbb{R}^3$. 
\begin{definition}[Multiphase surface network]\label{MultiNet}
A \emph{multiphase surface network} in $\mathbb{R}^3$ is a finite collection
\[
\Gamma=\bigcup_{i\in\mathtt{I}}\Gamma_i,\qquad \mathtt{I}=\{1,\dots,N\},
\]
of pairwise distinct, oriented $C^2$ interfaces $\Gamma_i \subset \mathbb{R}^3$, which represent the common boundaries of a family of disjoint open phase regions 
$\{\Omega_\alpha\}$ in $\mathbb{R}^3$. They satisfy the following conditions:
\begin{enumerate}[label=(\roman*)]
\item \textbf{Interfaces.}
Each $\Gamma_i$ is smoothly embedded away from its boundary, and any intersection between two distinct interfaces occurs only along their common boundary curves.
\item \textbf{Pairwise intersections.}
For any $i\neq j$, the intersection $\Gamma_i\cap\Gamma_j$ is either empty or a common boundary curve 
\[
\sigma_{ij} := \Gamma_i \cap \Gamma_j,
\]
which is a properly embedded $C^2$ curve in $\mathbb{R}^3$.  
Each $\sigma_{ij}$ is oriented consistently with the induced orientations of $\Gamma_i$ and $\Gamma_j$.
\item \textbf{$1$-junctions (line singularities).}
Each curve $\tau_{\mathtt{I}'}$ is the common boundary of exactly $\#\mathtt{I}'$ distinct interfaces 
\[
\tau_{\mathtt{I}'}:=\bigcap_{i\in\mathtt{I}'}\Gamma_i, \qquad \#\mathtt{I}'\ge 3,
\]
which meet smoothly along a common embedded curve $\tau_{\mathtt{I}'}$. These intersection curves are called $1$-junctions.
\item \textbf{$0$-junctions (point singularities).}
Endpoints of $1$-junctions where finitely many distinct line junctions meet are called $0$-junctions. Depending on the number of incident $1$-junctions, these points correspond to triple, quadruple, or higher-order junctions.
\item \textbf{Boundary.}
If the network lies inside a bounded domain $\Omega \subset \mathbb{R}^3$, 
its boundary $\partial \Gamma$ consists of those $1$-junction curves that either terminate on $\partial \Omega$ or are entirely contained in $\partial \Omega$. 
In the case of unbounded interfaces, $\partial \Gamma$ is understood as the set of curves which are endpoints of exactly one interface, corresponding to interfaces that cluster singly at infinity.
\end{enumerate}
\end{definition}

A multiphase surface network is called \emph{regular} if its junctions satisfy the additional regularity constraints dictated by area-minimizing conditions for minimal surfaces (cf.~\ref{1stVar}).
\begin{itemize}
\item[($\mathrm{iii}^\prime$)] \textbf{$1$-junction regularity.} 
Each $1$-junction curve $\tau_{\mathtt{I}'}$ must satisfy the \emph{equal-angle condition}, with dihedral angles $120^\circ$ between adjacent interfaces. Explicitly, exactly three interfaces meet along any line junction:
\[
\sum_{i=1}^3\vec{n}_i=0,
\]
where $\vec{n}_i$ are the normal vectors on their interfaces.
\item[($\mathrm{iv}^\prime$)] \textbf{$0$-junction regularity.} 
Each $0$-junction satisfies the \emph{geometric balancing condition}, forming a regular tetrahedral structure, and only quadruple points are allowed:
\[
\sum_{k=1}^{4} \vec{\nu}_k = 0
\]
where $\vec{\nu}_i$ is tangent vectors along the incident regular $1$-junctions $\tau_i$.
\end{itemize}
\begin{remark}
This definition generalizes the notion of surface clusters with triple edges introduced by Schulze and White~\cite{schulze2020local} to include regular quadruple points, in agreement with the equilibrium configurations of soap-film minimal surfaces studied by Taylor~\cite{taylor1976} and the calibrated structures analyzed by Lawlor~\cite{LawlorMorgan1994}.
\end{remark}

Let $\mathcal{C}_0 = \bigcup_{i=1}^N P_i$ be a finite union of planes meeting at the origin, 
so that $\bigcap_{i=1}^N P_i = \{0\}$ and the complement of $\mathcal{C}_0$ consists of finitely many open regions.
We impose the following geometric conditions:
\begin{itemize}
\item \textbf{Region structure.}  
Each region enclosed by $\mathcal{C}_0$ is bounded by $m$ planes with $2 \le m \le 5$.  
For each $i$, the intersection $P_i \cap \mathbb{S}^2$ defines a great-circle arc $\gamma_i$, so that
\[
\mathcal{C}_0 \cap \mathbb{S}^2 = \bigcup_{i=1}^N \gamma_i
\]
forms a geodesic network on the sphere.
\item \textbf{Junction structure.}  
Each intersection point of this geodesic network on $\mathbb{S}^2$ corresponds to a regular triple junction, where exactly three arcs meet at equal $120^{\circ}$ angles.
\end{itemize}

\begin{remark}
\label{LApexSolid}
Assume that four phases meet at a quadruple point with a locally tetrahedral configuration. 
In the case of a regular tetrahedral arrangement, the inscribed spherical cone at the quadruple point has a half-apex angle and the corresponding solid angle is
\begin{equation}\label{lowerSolid}
\theta = \cos^{-1}\!\left(\frac{1}{\sqrt{3}}\right),\qquad
\omega_{\min} = 2\pi\!\left(1 - \frac{1}{\sqrt{3}}\right).
\end{equation}
This inscribed spherical cone provides a geometric lower bound that prevents the local configuration from degenerating, 
ensuring that the solid angle at the quadruple point cannot collapse to zero under tetrahedral symmetry.
\end{remark}

Define the metric
\begin{align}\label{SelfEg}
g(\mathbf{x})=e^{\frac{|\mathbf{x}|^2}{4}}d\mathbf{x}^2,\qquad \mathbf{x}=(x_1, x_2, x_3)\in\mathbb{R}^3
\end{align}
which is complete and negatively curved. In the following argument, we let $\gamma_1,\cdots, \gamma_N$ be the prescribed boundary curves of $\mathcal{C}_0\cap \mathbb{S}^2$.

\begin{theorem}[Main Theorem]
\label{MainThm}
Let $\mathcal{C}_0 \subset \mathbb{R}^3$ be a finite union of half-planes meeting along lines through the origin, and let $\{\gamma_i\}_{i=1}^N \subset \mathbb{S}^2$ denote the corresponding intersection curves on the sphere at infinity.  
Then the family of self-similar solutions to the mean curvature flow emanating from $\mathcal{C}_0$, whose evolving surface network remains connected for all $t>0$, is in one-to-one correspondence with the set of (possibly disconnected) regular multiphase surface networks contained in the unit ball. Each interface is minimal with respect to the metric $g$ and the boundary of the surface network coincides with the prescribed asymptotic curves.
Moreover, there exists at least one (and generically finitely many) connected regular multiphase surface network realizing these asymptotic data.  
Finally, the corresponding self-similar solutions extend smoothly to a one-parameter family of regular networks on the parabolic blow-up of $\mathbb{R}^3 \times \mathbb{R}^+$ at the spacetime origin.
\end{theorem}

\section{Conformal Correspondence among the Hyperboloid, Poincar\'{e} Ball, and Euclidean Models}\label{ConformalCor}
In this section, we describe how the hyperboloid model of hyperbolic space naturally induces the conformal structure of the Poincaré ball model through stereographic projection in the Minkowski space $\mathbb{R}^{3, 1}$.
This correspondence provides a conformal equivalence between $\mathbb{B}^3$ and the Euclidean space $\mathbb{R}^3$. Such an identification enables a compact geometric representation of hyperbolic isometries on $\overline{\mathbb{B}^3}$. 
In particular, the conformal structure originally defined on the unbounded Euclidean domain $\overline{\mathbb{R}^3}:=\mathbb{R}^3 \cup \{\infty\}$ can be equivalently studied within the compact set $\overline{\mathbb{B}^3}:=\mathbb{B}^3 \cup \mathbb{S}^2_{\infty}$, providing a convenient framework for analyzing the existence and correspondence of conformal structures between the Euclidean and hyperbolic settings.

Let $\mathbb{R}^{3,1}$ denote the Minkowski space, that is, the four-dimensional vector space $\mathbb{R}^4$ endowed with the Lorentzian inner product $\langle\cdot,\cdot\rangle$.
For $\mathbf{x}=(x_1,x_2,x_3,x_4)$ and $\mathbf{y}=(y_1,y_2,y_3,y_4)$ in $\mathbb{R}^4$, this inner product is defined by
\begin{equation*}
\langle \mathbf{x}, \mathbf{y}\rangle := \sum_{i=1}^3 x_i y_i - x_4 y_4.
\end{equation*}
The induced metric tensor on $\mathbb{R}^{3,1}$ is therefore
\begin{equation*}
g_{_\mathrm{M}} := dx_1^2 + dx_2^2 + dx_3^2 - dx_4^2.
\end{equation*}
The hyperbolic space $\mathbb{H}^3$ is realized as the upper sheet of the two-sheeted hyperboloid:
\begin{equation*}
\mathbb{H}^3:=\left\{\mathbf{x}\in\mathbb{R}^{3, 1}\mid \mbox{$\langle \mathbf{x}, \mathbf{x}\rangle=-1$, $x_4\geq 1$}\right\}.
\end{equation*}
Analogous to the compactification of $\mathbb{R}^3$ by adding the point at infinity, there exists a corresponding compactification in $\mathbb{R}^{3,1}$.
Define the light cone
\begin{equation*}
L=\{\mathbf{l}\in\mathbb{R}^{3,1}\backslash\{\mathbf{0}\}\mid \langle \mathbf{l}, \mathbf{l}\rangle=0\},
\end{equation*}
on which the multiplicative group $\mathbb{R}^{\times}$ acts by scaling:
\begin{equation*}
\mathbb{R}^{\times}\times L\to L,\quad (r, \mathbf{l})\mapsto r\mathbf{l}.
\end{equation*}
The sphere at infinity of $\mathbb{H}^3$ is then given by the quotient space 
\[
\mathbb{S}^2_{\infty}=L/ \mathbb{R}^{\times}.
\]
\begin{definition}
The Poincar\'{e} ball model of hyperbolic space is the open submanifold
\begin{equation*}
\mathbb{B}^3:=\left\{\mathbf{u}=(u_1, u_2, u_3)\in\mathbb{R}^3\mid \left|\mathbf{u}\right|<1\right\},\label{ballmod}
\end{equation*}
equipped with the conformal metric 
\begin{equation*}
g_{_\mathrm{P}}=\frac{4}{\left(1-\left|\mathbf{u}\right|^2\right)^2}\left(du_1^2+du_2^2+du_3^2\right).
\end{equation*}
\end{definition}
In complete analogy with the conformal equivalence between $\mathbb{S}^2$ and $\mathbb{R}^2$ in complex analysis, the stereographic projection $\pi_{h}:\mathbb{H}^3\longrightarrow\mathbb{B}^3\subset\mathbb{R}^3$ is defined by,
\begin{align}\label{stereo}
\mathbf{x}=\left(x_1, x_2, x_3, x_4\right)\mapsto\mathbf{u}=\left(\frac{x_1}{1+x_4}, \frac{x_2}{1+x_4}, \frac{x_3}{1+x_4}\right).
\end{align}
Its inverse mapping is given by
\begin{equation}
\mathbf{u}=\left(u_1, u_2, u_3\right)\mapsto\mathbf{x}=\left(\frac{2u_1}{1-|\mathbf{u}|^2}, \frac{2u_2}{1-|\mathbf{u}|^2}, \frac{2u_3}{1-|\mathbf{u}|^2}, \frac{1+|\mathbf{u}|^2}{1-|\mathbf{u}|^2}\right).
\end{equation}
Geometrically, $\pi_h$ projects each point $\mathbf{x}\in\mathbb{H}^3$ onto $\mathbb{B}^3$ along the affine line passing through $(0,0,0,-1)\in\mathbb{R}^{3,1}$.
Both $\pi_h$ and $\pi_h^{-1}$ are diffeomorphisms, and under this correspondence, the Minkowski metric pulls back as
\begin{equation}\label{pullback}
g_{_\mathrm{M}}(d\mathbf{x}, d\mathbf{x})
= \langle d\mathbf{x}, d\mathbf{x} \rangle
= \pi_h^* g_{_\mathrm{P}}(d\mathbf{u}, d\mathbf{u}).
\end{equation}
Consequently, $(\mathbb{H}^3, g_{_\mathrm{M}})$ is isometric to $(\mathbb{B}^3, g_{_\mathrm{P}})$ and conformal to $(\mathbb{R}^3, \delta)$. 
Although the hyperboloid model $\mathbb{H}^3$ admits a simple linear isometric group action in $\mathbb{R}^{3,1}$, it is embedded as a hypersurface in four-dimensional space, which makes the visualization of its geometric structures less intuitive.
In contrast, the Poincar\'{e} ball model $\mathbb{B}^3$ is realized directly as a subset of $\mathbb{R}^3$, whose isometry group acts by M\"{o}bius (or linear fractional) transformations.
Hence, the Poincar\'{e} ball model provides a more convenient framework for describing the geometry of hyperbolic space.

If we regard $\mathbb{B}^3$ as a subset of $\mathbb{R}^3$, then the one-point compactification of $\mathbb{R}^3$ can be expressed as
\begin{equation}\label{onepoint}
\overline{\mathbb{R}^3}:=\mathbb{R}^3\cup\{\infty\}.
\end{equation}
On the other hand, the Poincar\'{e} ball model can be compactified by adjoining the ideal boundary $\mathbb{S}^2_{\infty}$ tangent to $\mathbb{B}^3$, i.e.,
\begin{equation}\label{idealbdry}
\overline{\mathbb{B}^3}:=\mathbb{B}^3\cup\mathbb{S}^2_{\infty}.
\end{equation}
From the stereographic correspondence established in \eqref{stereo}–\eqref{pullback}, we have that $\mathbb{R}^3$ is conformal to $\mathbb{B}^3$.
Moreover, by Proposition~6.17 in \cite{jensen2016}, there exists a conformal diffeomorphism between $\mathbb{S}^2$ and $\mathbb{S}^2_{\infty}$.
Therefore, the conformal geometry originally defined on the unbounded Euclidean domain $\overline{\mathbb{R}^3}=\mathbb{R}^3\cup\{\infty\}$ can be equivalently studied within the compact set $\overline{\mathbb{B}^3}=\mathbb{B}^3\cup\mathbb{S}^2_{\infty}$.
This identification allows a complete description of hyperbolic isometries in terms of conformal mappings on the boundary sphere.
In comparison with the hyperboloid model, the Poincar\'{e} ball model provides a more intuitive and compact geometric representation for analyzing conformal structures within the finite domain $\mathbb{B}^3\cup\mathbb{S}^2_{\infty}$, corresponding to those in the infinite domain $\mathbb{R}^3\cup\{\infty\}$.
Without loss of generality, we hereafter identify $\mathbb{S}^2$ in \eqref{onepoint} with $\mathbb{S}^2_{\infty}$ in \eqref{idealbdry}.
Further details on hyperbolic geometry can be found in \cite{BenedettiPetronio1992,jensen2016,Petersen2006}.

Under the conformal identification between the compactified Euclidean space 
$\overline{\mathbb{R}^3}$ 
and the compactified Poincar\'{e} ball model 
$\overline{\mathbb{B}^3}$, 
we introduce the following notion of the ideal boundary.
\begin{definition}
\label{IdeaBdry}
Let $\Gamma \subset \mathbb{R}^3$ be a smooth, complete noncompact embedded surface. The (ideal) boundary curve of $\Gamma$ on the sphere at infinity $\mathbb{S}^2_{\infty}$
is defined by the radial limit
\[
\sigma := \lim_{r\to\infty} \frac{1}{r}(\Gamma \cap \mathbb{S}^2_r),
\]
where $\mathbb{S}^2_r$ denotes the Euclidean sphere of radius $r$ centered at the origin. Equivalently, we may write $\sigma = \Gamma \cap \mathbb{S}^2_{\infty}$ to emphasize its geometric realization on the conformal boundary of the compactified space.
\end{definition}

\section{Self-expanding solutions to the multiphase mean curvature flow}\label{InitAsym}
Let $\Gamma \subset \mathbb{R}^{3}$ be a smooth hypersurface. A self-expanding solution to the mean curvature flow satisfies
\begin{equation}\label{eq:SE}
H = \frac{1}{2}\langle \mathbf{x}, \vec{n} \rangle,
\end{equation}
where $H$ is the mean curvature, $\mathbf{x}$ the position vector, and $\vec{n}$ the unit normal. Self-expanders arise as forward self-similar models of the mean curvature flow originating from a conical singularity.

Introducing the metric $g$ in~\eqref{SelfEg} transforms equation~\eqref{eq:SE} 
into the minimal surface equation associated to the weighted area functional~\eqref{entropyF}. The metric $g$ is complete with nonpositive sectional curvature decaying to zero at infinity; that is, it is asymptotically flat.  
Nevertheless, asymptotic flatness alone does not ensure good control of minimal surfaces at infinity: a minimal surface in $(\mathbb{R}^3, g)$ may oscillate or drift away from its prescribed asymptotic cone $\mathcal{C}_0$.

To analyze the behavior of each end of a self-expander, we examine the first variation on the planar sectors of the initial cone $\mathcal C_0$. Since every $1$-junction satisfies the (Plateau) balancing condition, the boundary terms along the
triple junction cancel when the variations of the adjacent sectors are summed. Thus, on each sector, the self-expander equation reduces to its interior part.

Because the (Plateau) boundary contributions cancel, the interior equation on each planar sector decouples at the level of the linearized operator. Consequently, in the far field, interactions between distinct planar ends become negligible, and each end may be treated independently as a graph over its asymptotic plane.

For the $i$-th plane end, the self-expander $\Gamma_i$ is graphical outside a compact set:
\[
\Gamma_i \cap \{\mathbf{x}\in\mathbb{R}^3 \mid|\operatorname{Proj}_{P_i}(\mathbf{x})|>R \} 
   = \{ \mathbf{x} + u_i(\mathbf{x}) \, \vec{n}_i \mid \mathbf{x} \in P_i \},\qquad R\gg 1,
\]
where $\vec{n}_i$ is the unit normal to $P_i$ and $u_i$ is its height function.  
Linearizing~\eqref{eq:SE} and computing the second variation of the weighted area functional 
yield the Jacobi operator governing the asymptotic behavior,
\[
L u_i := \Delta_{P_i} u_i + \frac{1}{2} \mathbf{x} \cdot \nabla u_i - \frac{1}{2} u_i.
\]
For asymptotically conical ends, the analysis of \cite{BernsteinWang2021} shows that
solutions of the Jacobi equation satisfy strong weighted energy estimates (Lemma~6.1) and admit a complete asymptotic expansion at infinity (Proposition~6.2). In particular, all non-decaying behaviors are excluded, and the height function $u_i(\mathbf{x})$ converges to zero in an appropriate weighted Hölder sense. Thus, no oscillatory behavior or deviation from the asymptotic plane can occur along any planar end.

This asymptotic control will be used in the next section to establish the existence of a global self-expanding solution asymptotic to the initial cone $\mathcal C_0$.

\section{Existence}\label{LocGloEx}
We consider the existence problem for a connected regular multiphase surface in $\mathbb{R}^3$ with prescribed boundary curves and junction structure $\mathcal{C}_0 \cap \mathbb{S}^2_\infty$.

To address this problem, we study the minimizing problem in the class $\mathcal{F}_2(\mathbb{R}^3, \mathbb{Z}_{K+1})$ of flat 2-chains $T$ with coefficients in $\mathbb{Z}_{K+1}$. Here, the nonzero coefficients encode the multiplicity of each face: the unit norm condition ensures that each face contributes exactly once to the size of the chain.
For a comprehensive treatment of flat chains and multiplicities, see \cite{white1996, morganGMT}. 
Below we briefly recall the key properties of flat chains relevant to the multiphase surface problem.

\smallskip
Here, the coefficients in $\mathbb{Z}_{K+1}$ encode the region assignment: $0$ represents points outside all regions, while nonzero elements $1,\dots,K$ correspond to the $K$ distinct regions.

\begin{itemize}
\item A flat 2-chain $T$ in $\mathcal{F}_2(\mathbb{R}^3, \mathbb{Z}_{K+1})$ is defined as the completion of the space of 2-dimensional polyhedral chains under the flat norm $\mathcal{F}$.
\item The norm of each nonzero coefficient $a_{T} \in \mathbb{Z}_{K+1}$ is defined to be one. The size of a flat chain $T$ is then defined by
\[
\operatorname{Size}(T) := \sum_\alpha |a_{T}| \, \mathbb{M}(T) = \sum_\alpha \mathbb{M}(\alpha),
\]
where $\mathbb{M}(T)$ denotes the mass of the chain $T$ with respect to the metric $g$ in~\eqref{SelfEg}. 
For $T$ a $C^1$-surface with unit multiplicity, the mass coincides with its area, so that $\operatorname{Size}(T)$ recovers the total area of the surface weighted by multiplicity.
\end{itemize}

We formulate the multiphase surface problem as a minimization problem for the size functional $\mathrm{Size}(T)$ over the class $\mathcal{F}_2(\mathbb{R}^3, \mathbb{Z}_{K+1})$.
To prescribe the boundary data and the phase structure, consider a finite union of planes
\[
\mathcal{C}_0 = \bigcup_{i=1}^N P_i
\]
passing through the origin, so that $\bigcap_{i=1}^N P_i = \{0\}$ and $\mathbb{R}^3\setminus \mathcal{C}_0$ decomposes into $K$ open regions.
For $R>0$, the induced boundary on the sphere of radius $R$ is
\begin{align*}
\gamma^R:=\mathcal{C}_0\cap \mathbb{S}^2_R(0)=\bigcup_{i=1}^{N}\gamma_i^R.
\end{align*}
\begin{theorem}
\label{LExist}
Given the boundary curves $\gamma^R$, there exists a connected flat $2$-chain $\Gamma^R$ with coefficients in $\mathbb{Z}_{K+1}$ such that
\begin{align*}
\operatorname{Size}(\Gamma^R)
= \inf_{T\in\mathcal{A}^R}\operatorname{Size}(T),
\end{align*}
where the admissible set 
\[
\mathcal{A}^R
:= \left\{\,T\in \mathcal{F}_2(\mathbb{B}^3_R(0),\mathbb{Z}_{K+1}) \mid
\partial T=\gamma^R,\ \operatorname{supp}(T)\ \text{connected}\,\right\}.
\]
In particular, $\Gamma^R$ attains the minimum.

Moreover, each sheet of the support $\Gamma^R$ is a minimal surface with respect to the metric:
\[
g(\mathbf{x})=e^{\frac{|\mathbf{x}|^2}{4}}d\mathbf{x}^2,\qquad \mathbf{x}=(x_1, x_2, x_3)\in\mathbb{R}^3.
\] 
and all interior $1$-junctions and $0$-junctions are regular triple curves and regular quadruple points, respectively.
\end{theorem}
\begin{proof}
Let $\{T_j\}\subset\mathcal{F}_2(\mathbb{B}^3_R(0),\mathbb{Z}_{K+1})$ be an area-minimizing sequence with connected supports such that $\partial T_j=\gamma^R$ for every $j$.
We verify the uniform bounds required to apply the compactness theorem for flat chains, and then deduce existence by lower semicontinuity.

\smallskip\noindent\textbf{Compactness for flat chains:}

Since $\{T_j\}$ is area-minimizing, there exists $c_1>0$ with
\[
\operatorname{Size}(T_j)\le c_1\qquad\text{for all }j.
\]
Depending only on $R$ and the geometry of $\mathcal{C}_0$, each $\gamma_i^R$ is a finite $\mathcal{H}^1$--rectifiable curve on the sphere $\mathbb{S}^2_R(0)$ and hence there is a constant $C(R)>0$ such that
\[
\operatorname{Size}(\partial T_j)\le \sum_{i=1}^N \mathcal{H}^1(\gamma_i^R)
\le C(R)\sum_{i=1}^N \mathcal{H}^1(\gamma_i),
\]
where $\gamma_i=\lim_{R\to\infty}\gamma_i^R/R$ is the asymptotic boundary determined in
Definition~\ref{IdeaBdry}. Hence, for each $R>0$, $\operatorname{Size}(\partial T_j)$ is uniformly bounded in $j$.

By the compactness theorem for flat chains (cf.~\cite{fleming1966}), from $\{T_j\}$ we may extract a subsequence (denoted again by $T_j$) converging in the flat topology to some
\(\Gamma^R\in\mathcal{F}_2(\mathbb{B}^3_R(0),\mathbb{Z}_{K+1})\) with the boundary \(\partial\Gamma^R=\gamma^R\). 
By lower semicontinuity of the size functional under flat convergence we obtain
\[
\operatorname{Size}(\Gamma^R)\le \liminf_{j\to\infty}\operatorname{Size}(T_j),
\]
so $\Gamma^R$ is area-minimizing in $\mathcal{F}_2(\mathbb{B}^3_R(0),\mathbb{Z}_{K+1})$.

Applying the regularity theorem (cf.~\cite[Theorem 2.6]{morgan1989}) for $m=2$ and $n=3$, a locally area-minimizing flat $2$-chain in a bounded domain has support which is a $C^{0,\alpha}$--surface away from the standard singularities: along $1$-dimensional curves three smooth sheets meet at equal $120^\circ$ angles, and at isolated points four such curves meet at the angle $\cos^{-1}(-1/3)$.

\smallskip\noindent\textbf{Connectedness:}

Suppose, for contradiction, that \(\Gamma^R=\Gamma_1+\Gamma_2\) is a nontrivial decomposition into two flat $2$-chains with disjoint supports, i.e.
\begin{equation}\label{NonDecomp}
\operatorname{supp}(\Gamma^R)=\operatorname{supp}(\Gamma_1)\cup\operatorname{supp}(\Gamma_2),\qquad 
\operatorname{supp}(\Gamma_1)\cap\operatorname{supp}(\Gamma_2)=\emptyset.
\end{equation}
Since \(\partial\Gamma^R=\gamma^R\), we have \(\partial\Gamma_1+\partial\Gamma_2=\gamma^R\).  
For the initial configuration \(\mathcal{C}_0\), the network \(\gamma^R=\mathcal{C}_0\cap\mathbb{S}^2_R(0)\) is connected, a union of great-circle arcs meeting pairwise at intersection points, so \(\gamma^R\) cannot be decomposed into two disjoint nonempty boundary pieces.  
Consequently at least one of \(\partial\Gamma_1,\partial\Gamma_2\) must vanish; without loss of generality assume \(\partial\Gamma_2=0\), so \(\Gamma_2\) is a nontrivial closed component disjoint from the prescribed boundary.

Then \(\Gamma_1=\Gamma^R-\Gamma_2\) is a competitor with the same boundary \(\gamma^R\). 
Moreover, since the supports are disjoint and all nonzero multiplicities equal one, the size is additive on this decomposition and \(\operatorname{Size}(\Gamma_2)>0\); therefore
\[
\operatorname{Size}(\Gamma_1)=\operatorname{Size}(\Gamma^R)-\operatorname{Size}(\Gamma_2)
<\operatorname{Size}(\Gamma^R),
\]
contradicting the minimality of \(\Gamma^R\). Hence, no nontrivial decomposition \eqref{NonDecomp} exists, and \(\operatorname{supp}(\Gamma^R)\) is connected.

\end{proof}

\begin{lemma}
\label{lem:persistence}
Assume the hypotheses of Theorem~\ref{LExist}, and let $\Gamma^r$ be an 
area--minimizing multiphase surface in $\mathbb B_r^3$ with initial boundary 
data induced by the cone $\mathcal{C}_0$. Then there exists a bounded neighborhood 
$\mathcal U\subset\mathbb R^3$ such that, for every radius $r>0$, all regular 
triple $1$-junctions and all quadruple points of $\Gamma^r$ lie inside 
$\mathcal U$, except for the endpoints of triple $1$-junctions that may escape 
to infinity along the rays of $\mathcal{C}_0$.
\end{lemma}
\begin{proof}
By the regularity theory for multiphase area--minimizers 
(cf.~\cite{taylor1976}), each minimizer $\Gamma^r\subset\mathbb B_r^3$ consists of smooth sheets meeting along a finite union of $C^{1,\alpha}$ triple $1$-junctions, with at most finitely many quadruple points. In particular, the singular set of $\Gamma^r$ is always nonempty in $\mathbb B_r^3$.

Singularities cannot disappear to infinity as $r\to\infty$. Indeed, if the singular set of $\Gamma^r$ escaped to infinity, then for every fixed $R>0$ the surfaces $\Gamma^r$ would be smooth in $\mathbb B_R^3$ for all sufficiently large $r$, contradicting both the prescribed conical boundary data at infinity and the non-smoothness of the limiting cone $\mathcal{C}_0$. This provides the intuitive mechanism preventing degeneration; the rigorous 
obstruction is supplied by the hyperbolic barrier constructions below. 

We work in the unit Poincar\'e ball model of hyperbolic space. By the conformal correspondence of Section~\ref{ConformalCor}, the $120^\circ$ balance condition at each regular triple $1$-junction is preserved under this representation.  
Using the hyperbolic angle-of-parallelism (Lobachevsky's formula), we may choose $\delta_1\in(0,1)$ 
such that for every point on the geodesic curves 
$\mathcal{C}_0\cap\mathbb S^2_{\delta_1}$, the interior angles in the totally geodesic 
plane orthogonal to the junction curve are strictly smaller than $120^\circ$.  
Hence no regular triple $1$-junction lying in $\mathbb B_{\delta_1}^3$ can 
collapse onto a common limiting geodesic on the sphere at infinity.  
Equivalently, each pair of adjacent sheets retains strictly positive 
separation, and the $120^\circ$ configuration persists.

Next, since $\mathcal{C}_0\cap\mathbb S^2$ has only finitely many regular triple points, 
we repeat the preceding construction along every junction curve meeting a 
triple point on the sphere.  
For each such point, the hyperbolic angle-of-parallelism estimate in the 
orthogonal totally geodesic section yields a uniform radius $\delta_2>0$ for 
which all interior angles in these sections remain strictly smaller than
\[
2\cos^{-1}\Bigl(\frac{1}{\sqrt{3}}\Bigr).
\]
By Remark~\ref{LApexSolid}, any quadruple point lying outside 
$\mathbb B_{\delta_2}^3$ would admit an inscribed spherical cone whose 
half-apex angle is strictly smaller than $\theta$, 
and thus whose solid angle is strictly less than $\omega_{\min}$.  
Such a configuration is incompatible with a regular quadruple point.  
Therefore every quadruple point of $\Gamma^r$ lies inside $\mathbb B_{\delta_2}^3$.

Set $\delta=\max\{\delta_1,\delta_2\}$ and define 
$\mathcal U=\pi_h^{-1}(\mathbb B_\delta^3)\subset\mathbb R^3$ via inverse 
stereographic projection.  
The preceding arguments show that every regular triple $1$-junction and every 
regular quadruple point of $\Gamma^r$ lies in $\mathcal U$, except for the terminal 
points of triple $1$-junction that may escape to infinity along the rays 
of $\mathcal{C}_0$. Hence the singular structures inside any fixed bounded regions are stable under increasing radii. Although the number or arrangement of singular structures of $\Gamma^r$ may vary with $r$, at least one nontrivial singular structure persists in every bounded region. This completes the proof.
\end{proof}

\begin{proof}[Proof of Theorem~\ref{MainThm}] 

Let $\{R_j\}_{j=1}^{\infty}$ be any sequence diverging to infinity. 
For each $R_j$, Theorem~\ref{LExist} provides a plane $2$-chain that minimises the area.
\[
\Gamma^{R_j} \in \mathcal{F}_2(\mathbb{B}^3_{R_j}(0), \mathbb{Z}_{K+1})
\]
spanning the boundary curves $\gamma_1^{R_j},\dots,\gamma_N^{R_j}$, with connected support. 

Fix $j$.  
Since each $\Gamma^{R_j}$ lies in the bounded domain $\mathbb{B}^3_{R_j}(0)$ and has uniformly bounded size on $\Gamma^{R_j}$ and $\partial \Gamma^{R_j}$, the compactness theorem for flat chains implies that any sequence of such chains admits a flat-convergent subsequence.  
Applying this to each $j$ and using a standard diagonal extraction, we obtain a sequence
\[
\Gamma_j^{R_j} \;\xrightarrow{\ \mathcal{F}\ }\; \Gamma
\quad\text{in }\mathbb{R}^3,
\]
where $\Gamma \in \mathcal{F}_2(\mathbb{R}^3,\mathbb{Z}_{K+1})$ is the flat limit of the truncated minimizers.

\smallskip\noindent\textbf{Junction persistence:}

Suppose, for contradiction, that a triple curve or quadruple point vanishes at infinity. 
By the conformal diffeomorphism between $\overline{\mathbb{R}^3}$ and $\overline{\mathbb{B}^3}$ in Section~\ref{ConformalCor}, this would correspond to approaching the ideal boundary of the Poincar\'{e} ball model. 
The barrier construction in Lemma~\ref{lem:persistence}, together with the density lower bounds of sheets ($\Theta=1$), triple curves ($\Theta=3/2$), and quadruple points ($\Theta=3\arccos(-1/3)/\pi>1$), prevents reaching $\mathbb{S}^2_\infty$, neck–pinching or collapse of junctions, while the multiphase $\mathbb{Z}_{K+1}$ labeling forbids sheet degeneration that would violate the prescribed boundary or global phase assignment.
Moreover, the multiphase structure encoded by coefficients in $\mathbb{Z}_{K+1}$ ensures that any degeneration of a sheet would alter the global phase assignment or boundary data, producing a nonminimal configuration that can be smoothed to reduce size. 
Therefore, all regular triple curves and quadruple points persist, and no sheet degenerates in the limit.

\smallskip\noindent\textbf{Connected support:}

To show that connectedness persists as $j\to\infty$, suppose for contradiction that the flat limit $\Gamma$ decomposes as
\[
\Gamma = \Gamma^{(1)} + \Gamma^{(2)}, \qquad 
\operatorname{supp}(\Gamma^{(1)}) \cap \operatorname{supp}(\Gamma^{(2)}) = \emptyset,
\]
with both components nontrivial. 

By Lemma~\ref{lem:persistence}, all regular triple curves and quadruple points are contained in $\Gamma$ and carry uniform density lower bounds in a fixed neighbourhood. Hence no junction can vanish or collapse into a thin neck. In particular, any nontrivial component of $\Gamma$ contains at least one such junction structure, which provides a uniform lower bound on the mass in that component.  

Since each truncated minimizer $\Gamma_j^{R_j}$ is connected by Theorem~\ref{LExist} and each component of $\Gamma$ contains a fixed amount of mass near its junctions, any decomposition of $\Gamma$ into disjoint nontrivial components would require the corresponding mass in $\Gamma_j^{R_j}$ to split as $j$ becomes large. This would contradict the connectedness of $\Gamma_j^{R_j}$, because no mass can disappear between junctions due to the density lower bounds and barrier arguments that prevent neck–pinching or escape to infinity.  

Therefore, $\Gamma$ cannot be decomposed into disjoint nontrivial components, and its support is connected.


\smallskip\noindent\textbf{Area-minimizing:}

Let \(T\) be any compactly supported competitor in \(\mathbb{R}^3\). 
Choose $j$ sufficiently large so that \(\operatorname{supp}(T) \subset \mathbb B^3_{R_j}(0)\).
For such $j$, the minimality of $\Gamma^{R_j}$ in $\mathbb B^3_{R_j}(0)$ implies
\[
\operatorname{Size}\big(\Gamma^{R_j}\big)\le \operatorname{Size}\big(T\cap \mathbb B^3_{R_j}(0)\big)
   = \operatorname{Size}(T).
\]
Taking $j\to\infty$ and using lower semicontinuity of the size functional on each compact set yields \(\operatorname{Size}(\Gamma) \le \operatorname{Size}(T)\).
Therefore $\Gamma$ is area-minimizing in $\mathbb{R}^3$ with respect to all compactly supported competitors.

Combining this with the asymptotic analysis in Section~\ref{InitAsym}, we conclude that $\Gamma$ remains asymptotic at infinity to the initial configuration $\mathcal{C}_0$ in the sense that the spherical projections \(\Gamma\cap\mathbb S^2_r\) converge in Hausdorff distance to \(\mathcal{C}_0\cap\mathbb S^2_r\) as \(r\to\infty\).
\end{proof}

\appendix
\section{First Variational of weighted area Functional}\label{1stVar}

Let $\Gamma_t \subset \mathbb{R}^3$ denote a smooth family of embedded interfaces at time $t>0$. 
We define the weighted area functional associated with a conformally weighted metric 
\begin{align*}
g(\mathbf{x},t)
= e^{\frac{C|\mathbf{x}|^2}{4\lambda^2(t)t}}
\,d\mathbf{x}^2,
\end{align*}
by
\begin{align}\label{entropyF}
\mathcal{F}_{g}(\Gamma_t)
:=\int_{\Gamma_t} 
e^{\frac{C|\mathbf{x}|^2}{4\lambda^2(t)t}}\,
d\mathcal{H}^{2}(\mathbf{x}),
\qquad  \mathbf{x}=(x_1,x_2,x_3)\in \Gamma_t,
\end{align}
where $C\in\mathbb{R}$ is a constant, $\lambda(t)>0$ is a smooth scaling function, 
and $d\mathcal{H}^2(\mathbf{x})$ denotes the surface measure on $\Gamma_t$ 
induced by the ambient Euclidean metric $d\mathbf{x}^2$. For general background concerning weighted area functionals, we refer to 
\cite{huisken1990, ilmanen1994, ColdingMinicozzi2012}.

For completeness, we note the following standard variational characterization.
\begin{proposition}
\label{criticalPt}
Let $\{\Omega_i\}_{i=1}^N$ be a multiphase partition of $\mathbb{R}^3$, and 
let $\Gamma_{ij}$ denote the interface between two distinct phase domains $\Omega_i$ and $\Omega_j$. If an embedded interface $\Gamma_{ij}$ is a regular self-expanding solution to the mean curvature flow, then it is a critical point of the weighted area functional \eqref{entropyF} with respect to normal variations.
\end{proposition}
\begin{proof}
Let $\alpha \subset \Gamma_{ij}$ be a smooth surface patch, and let 
$K \subset\subset U \subset \alpha$ be relatively compact in a local coordinate chart. 
Consider a smooth one-parameter family of diffeomorphisms 
$\{\Phi^s\}_{s\in[0,1]}$ on $U$, generated by a compactly supported vector field $X$ in $K$, satisfying
\[
\frac{d}{ds}\Phi^s(\mathbf{p}) = X(\Phi^s(\mathbf{p})), \qquad \Phi^0 = \mathrm{Id},\qquad \mathbf{p}\in U.
\]
The variation of the weighted area area functional \eqref{entropyF} on $\alpha\cap K$ is given by
\[
\mathcal{F}_g(\Phi^s(\alpha\cap K))
=\int_{\alpha\cap K}J\Phi^s_g\,d\mathcal{H}^2,
\qquad 
J\Phi^s_g=e^{\phi(\mathbf x,t)}J\Phi^s,
\]
where
\[
g=e^{\phi}d\mathbf x^2,\qquad 
\phi(\mathbf x,t)=\frac{C|\mathbf x|^2}{4\lambda^2(t)t}.
\]
Using standard first-variation computations (via the divergence theorem), we obtain
\begin{subequations}
\begin{align}\label{1varEq1}
\frac{d\mathcal F_g}{ds}\Big|_{s=0}
=\int_{\alpha\cap K}e^{\phi}
\left(\langle D\phi, X\rangle +\operatorname{div}_{\alpha}X\right) 
\,d\mathcal{H}^2.
\end{align}
Decomposing $X=X^\top+(X\cdot \vec{n})\vec{n}$ into its tangential and normal parts on $\alpha$ and using 
$\operatorname{div}_\alpha X=\operatorname{div}_\alpha X^\top - H(X\cdot \vec{n})$, 
the weighted divergence formula gives
\[
\int_{\alpha\cap K} e^{\phi}\operatorname{div}_\alpha X
= \int_{\partial(\alpha\cap K)} e^{\phi}\langle X,\nu\rangle\,d\mathcal H^1
- \int_{\alpha\cap K} e^{\phi}\langle\nabla_\alpha\phi,X\rangle\,d\mathcal H^2,
\]
where $\vec{\nu}$ denotes the outward co-normal along $\partial(\alpha\cap K)$. 
Hence
\begin{align}\label{1varEq2}
\frac{d\mathcal F_g}{ds}\Big|_{s=0}
=\int_{\partial(\alpha\cap K)}e^{\phi}\langle X, \vec{\nu}\rangle d\mathcal{H}^1
+\int_{\alpha\cap K}e^{\phi}\left(\langle\nabla^{\perp}\phi, X\rangle
-H(X\cdot \vec{n})\right)d\mathcal{H}^2.
\end{align}
\end{subequations}
If $X$ has compact support in $\alpha\cap K$, the boundary term in \eqref{1varEq2} vanishes. 
For normal variations $X=f\vec{n}$, we obtain
\[
\frac{d\mathcal F_g}{ds}\Big|_{s=0}
= \int_{\alpha\cap K} e^{\phi}\big(\langle\nabla^\perp\phi,\vec{n}\rangle - H\big) f \,d\mathcal H^2.
\]
The corresponding Euler–Lagrange equation for stationary points is
\begin{equation}\label{eq:EuL}
H = \langle\nabla^\perp\phi,\vec{n}\rangle,
\quad\text{i.e.}\quad 
\vec{H}=\frac{C}{2\lambda^2(t)t}\,\mathbf x^\perp.
\end{equation}
Thus, $\alpha$ is a self-expanding surface with the same normalization constant, showing that every self-expanding interface is a critical point of $\mathcal F_g$.

Finally, when the variation has support near a junction (triple or quadruple), one proceeds similarly but must handle the boundary integrals along the junction curves. 
If the variation is chosen compactly supported in a neighborhood of a triple junction and the junctions satisfy the balance condition on co-normal vectors $\vec{\nu}$ (or, equivalently, the sum of unit co-normals weighted by surface tensions vanishes), then the boundary terms cancel. 
An analogous statement holds at a quadruple point under balancing condition cited in~\cite{LawlorMorgan1994}. 
Hence under these balancing conditions the boundary terms vanish and the same Euler--Lagrange condition \eqref{eq:EuL} applies locally on each sheet, which completes the proof.
\end{proof}






\end{document}